\documentclass[11pt]{article}
\usepackage{graphicx}
\usepackage{amsfonts, amsthm}
\usepackage{amsxtra}
\usepackage[latin1]{inputenc}
\usepackage{fontenc}
\usepackage[dvips]{epsfig}
\usepackage[usenames]{color}
\begin{document}
\newtheorem{theo}{Theorem}[section]
\newtheorem{defi}[theo]{Definition}
\newtheorem{lemm}[theo]{Lemma}
\newtheorem{prop}[theo]{Proposition}
\newtheorem{rem}[theo]{Remark}
\newtheorem{exam}[theo]{Example}
\newtheorem{cor}[theo]{Corollary}
\newcommand{\mat}[4]{
    \begin{pmatrix}
           #1 & #2 \\
           #3 & #4
      \end{pmatrix}
   }
\def\Z{\mathbb{Z}} 
\def\R{\mathcal{R}} 
\def\I{\mathcal{I}} 
\def\C{\mathbb{C}} 
\def\N{\mathbb{N}} 
\def\PP{\mathbb{P}} 
\def\Q{\mathbb{Q}} 
\def\L{\mathcal{L}} 
\def\ol{\overline} 
\def\bs{\backslash} 
\def\part{P} 

\newcommand{\gcD}{\mathrm {\ gcd}} 
\newcommand{\End}{\mathrm {End}} 
\newcommand{\Aut}{\mathrm {Aut}} 
\newcommand{\GL}{\mathrm {GL}} 
\newcommand{\SL}{\mathrm {SL}} 
\newcommand{\PSL}{\mathrm {PSL}} 
\newcommand{\Mat}{\mathrm {Mat}} 

\newcommand\ga[1]{\overline{\Gamma}_0(#1)} 
\newcommand\pro[1]{\mathbb{P}1(\mathbb{Z}_{#1})} 
\newcommand\Zn[1]{\mathbb{Z}_{#1}}
\newcommand\equi[1]{\stackrel{#1}{\equiv}}
\newcommand\pai[2]{[#1:#2]} 
\newcommand\modulo[2]{[#1]_#2} 
\newcommand\sah[1]{\lceil#1\rceil} 
\def\sol{\phi} 
\begin{center}
{\LARGE\bf On conjugations of circle homeomorphisms with two
break points}  \\
\vspace{.25in} { \large {\sc Habibulla Akhadkulov\footnote{School of Mathematical Sciences, Faculty of Science and Technology, Universiti Kebangsaan, 43600 UKM Bangi, Selangor Darul Ehsan, Malaysia, \quad E-mail: akhadkulov@yahoo.com},
 Akhtam Dzhalilov\footnote{Faculty of  Mathematics and Mechanics, Samarkand State University, Boulevard st. 15, 703004 Samarkand, Uzbekistan, \quad E-mail: a\_dzhalilov@yahoo.com},

Dieter Mayer\footnote{Institut f\"ur Theoretische Physik, TU Clausthal, Leibnizstra{\ss}e 10, D-38678 Clausthal-Zellerfeld, Germany.
\quad E-mail: dieter.mayer@tu-clausthal.de}}}.

\end{center}


\title{conjugations between circle maps with a break point }

\begin{abstract}
Let   $f_{i}\in
C^{2+\alpha}(S^{1}\backslash\{{a_{i}},{b_{i}}\}) $, $\alpha>0, \,\,
i=1,2$, be circle homeomorphisms with two break points $a_{i}, b_{i}$ i.e. discontinuities in the
derivative $Df_{i} $, with identical irrational rotation number
$\rho$ and  $\mu_{1}([a_1,b_1])=\mu_{2}([a_2,b_2])$, where $\mu_i$ are the
 invariant measures of $f_i,\, i=1,2$. Suppose, the products of the jump ratios of $Df_{1} $ and $Df_{2} $
do not coincide, i.e.
 $\frac{Df_1(a_1-0)}{Df_1(a_1+0)}\cdotp\frac{Df_1(b_1-0)}{Df_1(b_1+0)}\neq
\frac{Df_2(a_2-0)}{Df_2(a_2+0)}\cdot\frac{Df_2(b_2-0)}{Df_2(b_2+0)}$
. Then
the
 map $ \psi$ conjugating $f_1$ and $f_2$ is a singular
function, i.e. it is continuous on $S^1,$ but $D\psi(x)=0$ a.e. with
respect to Lebesgue measure.
\end{abstract}

\section{Introduction}
 Let $f$ be an
orientation preserving homeomorphism of the circle
$S^{1}\equiv\mathbb{R}/\mathbb{Z}$ with lift
$\hat{f}:\mathbb{R}\rightarrow{\mathbb{R}}$, $\hat{f}$ continuous, strictly
increasing and $\hat{f}(t+1)=\hat{f}(t)+1$, $t\in\mathbb{R}$. The circle
homeomorphism $f$ is then defined by $f(x)=\hat{f}(\hat{x})$ $(mod\,1)$,
$x\in{S^{1}},$\, and $\,x\equiv \hat{x}+\mathbb{Z}$ with $\hat{x}\in [0,1)$. 
In the sequel $S^1$ will be identified with 
$[0,1)$ and $x \in S^1$ with $\hat{x} \in [0,1)$.
 The interval $[x,y] \subset S^1$
then corresponds to the interval $[\hat{x},\hat{y}]\subset [0,2)$.
 If $f$ is a circle diffeomorphism with irrational rotation number
$\rho=\rho_{f}$ and ${\log}\,D\hat{f}$ is of bounded variation, then $f$
is conjugate to the pure rotation $f_{\rho}$, that is, there
exists an essentially unique homeomorphism $\varphi$ of the circle
with $f=\varphi^{-1} \circ f_{\rho} \circ \varphi.$ This
classical result of Denjoy \cite{De1932} can be extended to circle
homeomorphisms with break points. The exact statement of the
corresponding theorem will be given later.

It is well known, that circle homeomorphisms $f$ with
irrational rotation number $\rho_{f}$ admit a unique $f$-
invariant probability measure $\mu_{f}$. Since the conjugating map
$\varphi$ and the invariant measure $\mu_{f}$ are related by
$\varphi(x)=\mu_{f}([0,x])$ (see \cite{CFS1982} ), regularity
properties of the conjugating map $\varphi$ imply
corresponding properties of the density of the absolutely
continuous invariant measure $\mu_{f}$. This problem of
smoothness of the conjugacy of smooth diffeomorphisms is by now
very well understood (see for instance
\cite{Ar1961,Mo1966,He1979,KO1989,KS1989,Yo1984}).

 An important
class of circle homeomorphisms are homeomorphisms with break
points or shortly, class P-homeomorphisms. In general their
ergodic properties like the invariant measures, their
renormalizations and also their rigidity properties are rather
different from those of diffeomorphisms (see
\cite{dMvS1993} chapter I and IV, \cite{He1979} chapter VI, \cite{
KhKm2003}).

The class of \textbf{ $P$-homeomorphisms} consists of orientation preserving
circle homeomorphisms $f$ whose lifts $\hat{f}$ are differentiable
away from countable many points $ \hat{b}\in BP(\hat{f})\subset [0,1)$, corresponding to the so called break points $b\in BP(f)\subset S^1$ of $f$, at which left and right derivatives, denoted
respectively by $D\hat{f}_{-}$ and $D\hat{f}_{+}$, exist,  such that
\begin{itemize}
\item[i)] there exist constants $0<c_{1}<c_{2}<\infty$ with
$c_{1}<D\hat{f}(\hat{x})<c_{2}$ for all $\hat{x}\in [0,1)\backslash{BP(\hat{f})}$,
$c_{1}<D\hat{f}_{-}(\hat{b})<c_{2}$ and $c_{1}<D\hat{f}_{+}(\hat{b})<c_{2}$ for
all $\hat{b}\in {BP(\hat{f})}$,
 \item[ii)] $\log\,D\hat{f}$ has bounded variation in $[0,1]$.
 \end{itemize}

 The ratio $\sigma_{f}(b):=\frac{D\hat{f}_{-}(\hat{b})}{D\hat{f}_{+}(\hat{b})}$ is
called the \textbf{jump ratio} of $f$ in $b\in BP(f)$. 
 Denote by $v$
the totaol variation  $v=Var_{[0, 1]}(\log\, D\hat{f})$ of $\log\,D\hat{f}$ on $[0, 1]$

General  $P$-homeomorphisms with
one break point  were first studied by K. Khanin and E. Vul in
\cite{KV1991}. Among other results it was proved by these authors that their  renormalizations  approximate
fractional linear transformations.
 Piecewise linear $(PL)$ orientation preserving circle
homeomorphisms with  break points are the simplest examples in the
class of P-homeomorphisms. They show up in many other areas of
mathematics as for instance in group theory, homotopy theory and in
logic via the Thompson group and its generalizations (see
\cite{St1992}). The invariant
measures of $PL$ homeomorphisms were first studied by M. Herman in
\cite{He1979}, those of general  $P$-homeomorphisms with
one break point by A. Dzhalilov and K. Khanin in
\cite{DK1998}. Their  main result
 is the following 

\begin{theo}\label{DK} Let $f$ be a  $P$-homeomorphism with one
break point $b$.
If the rotation number $\rho_{f}$ is irrational and $f\in
C^{2+\varepsilon}(S^{1}\backslash\{b\})$
for some $\varepsilon>0$, then the $f$-invariant probability measure
$\mu_{f}$ is singular with
 respect to Lebesgue measure $\mu_L$ on $S^{1}$, i.e. there exists a
 measurable subset $A\subset{S^{1}}$ such that $\mu_{f}(A)=1$ and
 $\mu_L(A)=0$.
\end{theo}
I. Liousse got in \cite{Li2005} the same result for "generic"
$PL$ circle homeomorphisms with several break points whose
rotation number is irrational and of bounded type. In a next step A. Dzhalilov and I. Liousse
 \cite{DL2006} and A. Dzhalilov, I. Liousse and D. Mayer  \cite{DLM2009} studied
 another class of circle homeomorphisms with two break points.
Their  main result in  \cite{DLM2009}  is
\begin{theo}\label{1.2} Let $f$ be a  $P$-homeomorphism
satisfying the following conditions:
\begin{itemize}
\item[(a)] the rotation number $\rho=\rho_{f}$ of $f$ is
irrational;
 \item[(b)] $f$ has two break points $b_1$,
$b_2$ with $\sigma_{f}(b_1)\cdot\sigma_{f}(b_2)\ne 1$;
\item[(c)] $D\hat{f}$ is absolutely continuous on every connected
interval of $[0,1]\backslash\{\hat{b_1}, \hat{b_2}\}$ and
 \\$D^{2}\hat{f}\in{L^{1}([0,1],d\mu_L)}$.
\end{itemize}
Then the $f$- invariant probability measure $\mu_{f}$ is
singular with respect to Lebesgue measure $\mu_L$.\end{theo}

In the sequel we refer to the smoothness condition (c)  in Theorem 1.2 on $f$ as the Katznelson-Ornstein (KO)  condition.

The above theorems show that
for a  sufficiently piecewise smooth  circle homeomorphism $f$ with irrational rotation number and one or two break points
the map conjugating  $f$ and $f_{\rho}$ is singular. Consider next the regularity properties of the conjugating map between two class $P$-homeomorphisms with
one or two break points and coinciding irrational rotation numbers. The case of one break point with
coinciding jump ratios, the so called rigidity problem, was studied in great detail
 by K. Khanin and D. Khmelev in  \cite{KhKm2003} and by A. Teplinskii and
K. Khanin in   \cite{TK2004}.

If $\rho=[k_1,k_2,\ldots,k_n,\ldots]$ is the
continued fraction expansion of the irrational rotation number
$\rho$,  define the sets
$$
M_o=\{\rho:\, \forall n\in \mathbb{N}\;\exists \, C>0 :\, k_{2n-1}\leq
C\},
$$
$$
M_e=\{\rho: \forall n\in \mathbb{N} \;\exists \, C>0: \, k_{2n }\leq
C\}.
$$ The main result of \cite{TK2004}  is then  the following
\begin{theo}\label{ADM0}\emph({Teplinskii-Khanin}).
Let  $f_{i}\in C^{2+\alpha}(S^{1}\backslash\{{b_{i}}\}), i=1,2$,  be $P$- homeomorphisms each with one break
point $b_i $.  Assume 
\begin{itemize}
\item[(1)] their rotation numbers $\rho(f_{i}),\, \ i=1,2,$ are irrational and
coincide, i.e. $\rho(f_{1})=\rho(f_{2})=\rho, \
\rho\in\mathbb{R}^{1}\setminus\mathbb{Q}$;
 \item[(2)] their jump ratios
$\sigma_i=\sigma_{f_{i}}(b_{i}), \,i=1,2,$ coincide, i.e.
$\sigma_1=\sigma_2=\sigma$.
\end{itemize}
Then the map $\psi$ conjugating the homeomorphisms $f_ 1$ and
$f_2$ is a $C^1-$ diffeomorphism of the circle if either $\sigma> 1$
and  $\rho\in M_o$ or $\sigma<1$ and $\rho\in M_e$.\\
\end{theo}

In the case of not coinciding jump ratios A. Dzhalilov, H. Akin and S. Temir \cite{DHT2010} proved  

\begin{theo}\label{DzAT} Let  $f_{i}\in C^{2+\alpha}(S^{1}\backslash\{{b_{i}}\}),  \ i=1,2$,
 be $P$- homeomorphisms each with one break
point $b_i $.  Assume 
\begin{itemize}
\item[(1)] their rotation numbers $\rho_{i}, \, i=1,2,$ are irrational and
coincide i.e. $\rho_{1}=\rho_{2}=\rho, \,\rho
\in \mathbb{R}^{1}\setminus\mathbb{Q}$;
 \item[(2)] their jump ratios
$\sigma_{f_{i}}(b_i), \, i=1,2,$ are positive and do not
coincide.
\end{itemize}
Then the homeomorphism $\psi$ conjugating $f_{1}$ and $f_{2}$ is a
singular function, i.e. $\psi$ is continuous on $S^{1}$
and $D\psi(x)=0$ a.e. with respect to Lebesgue measure  $\mu_L$.\\
\end{theo}

In the present paper we will extend this result to  circle homeomorphisms with
coinciding irrational rotation numbers having each two break points.
Our main result  is  the following
\begin{theo}\label{ADM1}
 Let  $f_{i}\in C^{2+\alpha}(S^{1}\backslash\{{a_{i},b_{i}}\}), i=1,2 ,$
be $P$- homeomorphisms each with two break points $ a_{i},b_{i}.$ Assume 

\begin{itemize}
\item[(1)] their rotation numbers $\rho(f_{i}), \ i=1,2,$ are irrational and
coincide i.e. $\rho(f_{1})=\rho(f_{2})=\rho, \
\rho\in\mathbb{R}^{1}\setminus\mathbb{Q}$;

\item[(2)] the products of their jump ratios $\sigma_{f_{i}}(a_{i})\cdot\sigma_{f_{i}}(b_{i})$ do not coincide i.e.
$\sigma_{f_{1}}(a_{1})\cdot\sigma_{f_{1}}(b_{1})\neq \sigma_{f_{2}}(a_{2})\cdot\sigma_{f_{2}}(b_{2});$

\item[(3)]
$\mu_{1}([a_{1},b_{1}])=\mu_{2}([a_{2},b_{2}]),$ where $\mu_{i}$
is the invariant probability measure of $f_{i}, i=1,2$.
\end{itemize}
Then the  map $\psi$ conjugating $f_{1}$ and $f_{2}$ is singular.
\end{theo}

\section{Preliminaries and Notations}

Consider  an orientation preserving circle homeomorphism $f$
with lift $\hat{f}$ and irrational rotation number
$\rho=\rho_{f}$. 
If the rotation number $\rho$ has the 
continued fraction expansion
$\rho=\left[k_{1},k_{2},...,k_{n},...\right]=
1/\left(k_{1}+ 1/\left( k_{2}+...+1/\left(k_{n}+...\right) \right)
\right) $ its convergents $p_n/q_n,\, n \in \mathbb{N},$ are defined by
$p_{n}/q_{n}=[k_{1},k_{2},...,k_{n}] $.
Then the denominators $q_{n}$ satisfy the well known recursion
relation $q_{n+1}=k_{n+1}q_{n}+q_{n-1}, \ n\geq 1, \ q_{0}=1, \
q_{1}=k_{1}$.

For an arbitrary point $x_{0}\in S^{1}$ define
$\Delta_{0}^{(n)}(x_{0})$ as the closed interval in $S^1$ with endpoints
$x_{0}$ and $x_{q_{n}}=f^{q_n}(x_0)$, such that for $n$ odd $x_{q_{n}}$ is to the
left of $x_{0}$ and for $n$ even it is to its right with respect to the orientation induced from the real line. Denote by
$\Delta_{i}^{(n)}(x_{0}):=f^{i}(\Delta_{0}^{(n)}(x_{0})), i \ge 1$, the
iterates of the
interval $\Delta_{0}^{(n)}(x_{0})$ under $f$.
It is well known, that the set $\xi_n(x_0)$ of
intervals with mutually disjoint interiors defined as
\begin{equation}\label{eq2}
\xi_{n}(x_{0})=\left\lbrace \Delta_{i}^{(n-1)}(x_{0}), \ 0\leq i<q_{n}\right\rbrace \cup
\left\lbrace \Delta_{j}^{(n)}(x_{0}), \ 0\leq j<q_{n-1}\right\rbrace
\end{equation}
determines a partition of the circle for any $n$.
The partition $\xi_{n}(x_{0})$ is called the $n$-th \textbf{dynamical
partition} of the point
$x_{0}$ with \textbf{generators} $\Delta_{0}^{(n-1)}(x_{0})$ and
$\Delta_{0}^{(n)}(x_{0})$. Obviously, the partition $\xi_{n+1}(x_0)$ is a
refinement of the partition $\xi_{n}(x_{0})$: indeed the intervals of
order $n$ belong to $\xi_{n+1}(x_0)$ and each interval $
\Delta_{i}^{(n-1)}(x_{0})\in \xi_{n}(x_{0}), \, \ 0\leq
i<q_{n},$ is partitioned into $k_{n+1}+1$ intervals belonging to
$\xi_{n+1}(x_0)$ such that
\begin{eqnarray*}
\Delta_{i}^{(n-1)}(x_{0})=\Delta_{i}^{(n+1)}(x_{0})\cup\bigcup_{s=0}^{k_{n+1}-1}\Delta_{i+q_{n-1}+sq_{n}}^{(n)}(x_{0}).
\end{eqnarray*}

Recall the following definition introduced in \cite{KO1989}:
\begin{defi}\label{defKO} An interval
$I=(x,y)\subset S^1$ is $q_{n}$-small and its endpoints
$x,y$ are $q_{n}$-close if the intervals
$f^{i}(I),\ 0\leq i<q_{n},$ are disjoint.
\end{defi}

It is clear that the interval $(x,y)$ is $q_{n}$-small if,
depending on the parity of $n,$ either $y\prec x\preceq
f^{q_{n-1}}(y)\prec y\; \text{or}\;  f^{q_{n-1}}(x)\preceq
y\prec x\prec f^{q_{n-1}}(x)$ in the order induced from the real line.

Then we can show
\begin{lemm}\label{lemm2.2} Let $f$ be a P-homeomorphism with a
finite number
 of break points $b_i, i=1,2,...,m$, and irrational rotation
number $\rho$. Assume
 $x,y\in S^{1}$ are $q_{n}$-close and $b_i\notin\left\lbrace
f^{j}(x),\,\,f^{j}(y), \,0\leq j<q_{n},\right\rbrace $,\,$i=1,2,...,m$. Then for
any $0\leq k < q_{n}$ the following inequality holds:
\begin{equation}\label{eq5}
e^{-v}\leq \frac {D\hat{f}^{k}(\hat{x})}{D \hat{f}^{k}(\hat{y})}\leq e^{v}.
\end{equation} where $v$ is the total variation of $\log D\hat{f}$ on $[0,1]$ and
$\hat{x},\,\hat{y}$ are the lifts of $x,\,y$ to the interval $[0,1).$
\end{lemm}
\begin{proof} Take any two $q_{n}$-close points $x,y\in S^{1}$
and $0\leq k< q_{n}$. Denote by $I$ the open interval with
endpoints $x$ and $y$. Because the intervals $f^{s}(I),\ 0\leq
s<k$ are disjoint, we obtain
\begin{displaymath}
|\log D\hat{f}^{k}(\hat{x})-\log D\hat{f}^{k}(\hat{y})|\le \sum_{s=0}^{k-1}|\log
D\hat{f}(\hat{f}^{s}(\hat{x}))-\log D\hat{f}(\hat{f}^{s}(\hat{y}))|\le v,
\end{displaymath}
from which inequality (\ref{eq5}) follows immediately.
\end{proof}
The following Lemma can be proven easily using  Lemma \ref{lemm2.2}.
\begin{lemm}\label{lem2.1} Let $f$ be a P-homeomorphism with a
finite number
 of break points $b_i, \, i=1,2,...,m,$ and irrational rotation
number $\rho$.
If $x_{0}\in S^{1},$ $ n\ge 1$ and $b_i\notin\left\lbrace
f^{j}(x_{0}), 0\leq j<q_{n}\right\rbrace $ for $i=1,2,...,m$, then
\begin{equation}\label{eq3}
e^{-v}\leq \prod_{i=0}^{q_{n}-1}D\hat{f}(\hat{f}^{i}(\hat{x}_{0}))\leq e^{v}.
\end{equation}

\end{lemm}

Inequality (\ref{eq3}) is called the \textbf{Denjoy inequality}.
The proof of Lemma 2.3 is as for circle diffeomorphisms
(see for instance \cite{KS1989}). Using Lemma 2.3 it can
be shown that the intervals of the dynamical
partition $\xi_{n}(x_0)$ in (\ref{eq2}) have exponentially small length .
Indeed one finds

\begin{cor}\label{cor2} Let $\Delta^{(n)}$ be an arbitrary element of
 the dynamical partition $\xi_{n}(x_{0})$. Then
\begin{equation}\label{eq4}
l(\Delta^{(n)}):=\mu_L(\Delta^{(n)}) \leq const \,\, \lambda^{n},
\end{equation}
where $\lambda=(1+e^{-v})^{-1/2}<1.$
\end{cor}
From Corollary \ref{cor2} it follows that the trajectory of every
point $x\in S^1$ is dense in $S^{1}$. This together with
monotonicity of the homeomorphism $f$ implies the following

\begin{theo}\label{Denjoy} Suppose that a homeomorphism $f$ satisfies
the conditions of Lemma 2.3. Then $f$ is topologically
conjugate to
the linear rotation $f_{\rho}$.
\end{theo}
In the following discussion we have to compare
different intervals. For this we use
\begin{defi}\label{Def2.3} {\rm
Let $C>1$. We call two intervals of $S^{1}$ \textbf{C-comparable } if
the ratio of their lengths is in $[C^{-1}, C].$}
\end{defi}

 \begin{cor}\label{cor3}  Suppose the homeomorphism $f$ satisfies the conditions of
Lemma 2.3. Then for any interval $I\subset S^{1}$ the intervals
$I$ and $f^{q_{n}}(I)$ are $e^{v}$-comparable. If the interval $I$
is $q_{n}-$small then $l(f^{i}(I))< const \, \lambda^{n} $ for all
$i=0,1, ...,q_{n}-1$.
\end{cor}
\section{The Cross-ratio Tools}
Let us first recall two definitions:
\begin{defi}\label{crossratio} {\rm The \textbf{cross-ratio }
$Cr(a_{1},a_{2},a_{3},a_{4})$ of four strictly ordered points $a_{i}\in \mathbb{R}$, $i=1,2,3,4$,
 is defined as
$$
Cr(a_{1},a_{2},a_{3},a_{4})=\frac{(a_{2}-a_{1})(a_{4}-a_{3})}{(a_{3}-a_{1})(a_{4}-a_{2})}.
$$}
\end{defi}
\begin{defi}\label{distortion} {\rm
The \textbf{cross-ratio distortion} $Dist(a_{1},a_{2},a_{3},a_{4};f)$ of
four strictly ordered points
 $a_{i}\in \mathbb{R}$, $i=1,2,3,4$
with respect to a strictly increasing function $f$ on $
\mathbb{R}$ is defined as
$$
Dist(a_{1},a_{2},a_{3},a_{4};f)=\frac{Cr(f(a_{1}),f(a_{2}),f(a_{3}),f(a_{4}))}
{Cr(a_{1},a_{2},a_{3},a_{4})}.
$$}
\end{defi}
 For $k\geq 3$ let $\hat{z}_{i}\in [a,a+1]\subset \mathbb{R} \,\,,\,i=1,{...},k$ be the lifts of the points $z_{i}\in S^{1} , i=1,{...},k, $ with  
  $z_{1}\prec{z_{2}}\prec{...}\prec{z_{k}}\prec{z_{1}}$ such that  $\hat{z}_{1}<{\hat{z}_{2}}<{...}<{\hat{z}_{k}}.$
 The vector $(\hat{z}_{1},\hat{z}_{2},...,\hat{z}_{k})\in \mathbb{R}^{k} $
is called the \textbf{lifted vector} of $(z_{1},z_{2},...,z_{k})\in (S^1)^{k}$.
 Consider a circle homeomorphism $f$ with lift $\hat{f}$. We
 define the cross-ratio distortion of $(z_{1},z_{2},z_{3},z_{4})$
 with respect to $f$ by
 $Dist(z_{1},z_{2},z_{3},z_{4};f):=Dist(\hat{z}_{1},\hat{z}_{2},\hat{z}_{3},\hat{z}_{4};\hat{f})$
 where $(\hat{z}_{1},\hat{z}_{2},\hat{z}_{3},\hat{z}_{4})$ is the
 lifted vector of $(z_{1},z_{2},z_{3},z_{4})$. We need the following
\begin{lemm}\label{lemm3.1}\emph{(see \cite{DK1998})}
 Suppose $f$ is a $P$-homeomorphism with a finite number of break points and
 $f\in C^{2+\alpha}(S^{1}\backslash BP(f))$ for some $\alpha>0$.  Consider
  any four points
$z_{i}\in{S^{1}}$, $i=1,2,3,4$, with
$z_{1}\prec{z_{2}}\prec{z_{3}}\prec{z_{4}}\prec{z_{1}}$ and $[z_{1},z_{4}]\subset{S^{1}\backslash BP(f)} $
Then
$$|Dist(z_{1},z_{2},z_{3},z_{4};f)-1|\leq{K|\hat{z}_{4}-\hat{z}_{1}|^{1+\alpha}}$$
for some 
positive constant $K$ depending only on $f$.
\end{lemm}
Next we consider the case when the interval $[z_{1},z_{4}]$
contains one break point $b$ of the homeomorphism $f$. We
estimate the distortion of its cross-ratio when $b$ lies outside
the middle interval $[z_{2},z_{3}]$. For this we define for $z_{i}\in{S^{1}},$ $i=1,2,3,4,$ with
$z_{1}\prec{z_{2}}\prec{z_{3}}\prec{z_{4}}\prec{z_{1}}$ and
$b\in{[z_{1},z_{2}]\cup[z_{3},z_{4}]}$ the following quantities:
$$\alpha:=\hat{z}_{2}-\hat{z}_{1}, \,\beta:=\hat{z}_{3}-\hat{z}_{2},\,
\gamma:=\hat{z}_{4}-\hat{z}_{3},\,\tau:=\hat{z}_{2}-\hat{b},\,\xi:=\frac{\beta}{\alpha}, \,\zeta:=
\frac{\tau}{\alpha},\,\eta:=\frac{\beta}{\gamma},\,\vartheta:=\frac{\hat{b}-\hat{z}_{3}}{\gamma}.$$

\begin{lemm}\label{lemm3.4} Assume $f$  is $P$-homeomorphism with  a finite number of break points and $f\in C^{2}(S^{1}\setminus BP(f))$.
Choose points $z_{i}\in{S^{1}}$, $i=1,2,3,4$,
with
$z_{1}\prec{z_{2}}\prec{z_{3}}\prec{z_{4}}\prec{z_{1}}$ such that $f$ has one
single break point $b$ in   $[z_{1},z_{2}]\cup[z_{3},z_{4}] $.
 Then
\begin{itemize}
 \item[i)] $|Dist(z_{1},z_{2},z_{3},z_{4};f)-\frac{[\sigma_{f}(b)+(1-\sigma_{f}(b))\zeta](1+\xi)}
     {\sigma_{f}(b)+(1-\sigma_{f}(b))\zeta+\xi}|\leq
K_{1}|\hat{z}_{4}-\hat{z}_{1}|,
     $ if \,\, $b \in {[z_{1},z_{2}]}$,
 \item[ii)]
$|Dist(z_{1},z_{2},z_{3},z_{4};f)-\frac{[\sigma_{f}(b)+(1-\sigma_{f}(b))\vartheta](1+\eta)}
{\sigma_{f}(b)+(1-\sigma_{f}(b))\vartheta+\eta}|\leq
K_{1}|\hat{z}_{4}-\hat{z}_{1}|,
     $ if \,\, $b\in{[z_{3},z_{4}]}$

\end{itemize}
for some positive constant $K_{1}$ depending only on $f$.
\end{lemm}
\begin{proof}
We prove only the first assertion of  Lemma \ref{lemm3.4}. The second one can be proved similarly. Obviously
$$
\hat{f}(\hat{z}_{2})-\hat{f}(\hat{z}_{1})=D\hat{f}_{+}(\hat{b})
(\hat{z}_{2}-\hat{b})+\frac{D^{2}\hat{f}(\theta_{1})(\hat{z}_{2}-\hat{b})^{2}}{2}+
D\hat{f}_{-}(\hat{b})
(\hat{b}-\hat{z}_{1})+\frac{D^{2}\hat{f}(\theta_{2})(\hat{b}-\hat{z}_{1})^{2}}{2}
$$
and
$$
\hat{f}(\hat{z}_{3})-\hat{f}(\hat{z}_{1})=D\hat{f}_{+}(\hat{b})
(\hat{z}_{3}-\hat{b})+\frac{D^{2}\hat{f}(\theta_{3})(\hat{z}_{3}-\hat{b})^{2}}{2}+
D\hat{f}_{-}(\hat{b})
(\hat{b}-\hat{z}_{1})+\frac{D^{2}\hat{f}(\theta_{4})(\hat{b}-\hat{z}_{1})^{2}}{2},
$$
for some  $\theta_{1}\in(\hat{b}, \hat{z}_{2}),$ \, $\theta_{2}\in(\hat{z}_{1}, \hat{b}),$ \, $\theta_{3}\in(\hat{b}, \hat{z}_{3}),$ \,$\theta_{4}\in(\hat{z}_{1}, \hat{b}).$

Using the last two relations it is easy to show
\begin{eqnarray}\label{eq411}
\frac{\hat{f}(\hat{z}_{2})-\hat{f}(\hat{z}_{1})}
{\hat{f}(\hat{z}_{3})-\hat{f}(\hat{z}_{1})}=
\frac{\sigma_{f}(b)+(1-\sigma_{f}(b))\zeta+O(\alpha)}{G(\zeta, \xi)+O(\alpha+\beta)}\times \frac{\alpha}{\alpha+\beta}
\end{eqnarray}
where $G(\zeta, \xi)=(\sigma_{f}(b)+(1-\sigma_{f}(b))\zeta +\xi)/(1+\xi)$ and $\xi>0.$
It is clear that
$\min\{1, \sigma_{f}(b)\}\leq \sigma_{f}(b)+(1-\sigma_{f}(b))\zeta \leq \max\{1, \sigma_{f}(b)\}$ and
$\min\{1, \sigma_{f}(b)\}\leq G(\zeta, \xi) \leq 1+\max\{1, \sigma_{f}(b)\}.$
The last two inequalities together with (\ref{eq411}) imply that
\begin{eqnarray}\label{eq412}
\frac{\hat{f}(\hat{z}_{2})-\hat{f}(\hat{z}_{1})}
{\hat{f}(\hat{z}_{3})-\hat{f}(\hat{z}_{1})}:\frac{\alpha}{\alpha+\beta}=
\frac{[\sigma_{f}(b)+(1-\sigma_{f}(b))\zeta](1+\xi)}{\sigma_{f}(b)+(1-\sigma_{f}(b))\zeta
+\xi}+O(\alpha+\beta).
\end{eqnarray}
Since $\hat{f}\in C^{2}([\hat{z}_{2}, \hat{z}_{4}]),$ we get
\begin{eqnarray}\label{eq413}
\frac{\hat{f}(\hat{z}_{4})-\hat{f}(\hat{z}_{3})}
{\hat{f}(\hat{z}_{4})-\hat{f}(\hat{z}_{2})}:\frac{\gamma}{\gamma+\beta}=
1+O(\gamma+\beta).
\end{eqnarray}
The relations (\ref{eq412}) and (\ref{eq413}) imply the first assertion of Lemma \ref{lemm3.4}. The second one can be proved similarly.
\end{proof}

\section {Proof of Theorem \ref{ADM1}.}

For the proof of Theorem \ref{ADM1} we need several Lemmas which we
formulate next. Their proofs will be given later. Consider two
copies of the circle $S^{1}$ and homeomorphisms $f_{i}$ each with two
break points $a_{i},  b_{i}, \, i=1,2$, and the same irrational rotation number $\rho$.
 Assume that $f_{1}$ and $f_{2}$ satisfy the conditions of
Theorem \ref{ADM1}.
\par Let
$\varphi_{1}$ and $\varphi_{2}$ be maps conjugating $f_{1}$ and
$f_{2}$ with the pure rotation $f_{\rho}$, i.e. $\varphi_{1}\circ
f_{1}=f_{\rho}\circ\varphi_{1}$ and $\varphi_{2}\circ
f_{2}=f_{\rho}\circ \varphi_{2}$. It is easy to check that the map
$\psi=\varphi^{-1}_{2}\circ \varphi_{1}$ conjugates $f_{1}$ and
$f_{2}$ , i.e.
\begin{eqnarray}\label{eq41}
\psi(f_{1}(x))=f_{2}(\psi(x))
\end{eqnarray}
for all $x\in S^{1}$. By assumption in  Theorem \ref{ADM1}
 $\sigma_{f_{1}}(a_{1})\cdot\sigma_{f_{1}}(b_{1})\neq\sigma_{f_{2}}(a_{2})\cdot \sigma_{f_{2}}(b_{2}).$
 W.l.o.g assume
 $\sigma_{f_{1}}(a_{1})\neq\sigma_{f_{2}}(a_{2}).$
Since $\varphi_{i}, \ i=1,2,$ is unique up
to an additive constant we can choose $\varphi_{i}, \
i=1,2,$ such that $\varphi_{1}(a_{1})=a_{1}$ and  $\varphi_{2}^{-1}(a_{1})=a_{2}$ and hence $\psi(a_{1})=a_{2}$.
Then by assumption of Theorem \ref{ADM1}  $\psi(b_{1})=b_{2}.$
 Recall, that the
length of an interval $[a,b]\subset{S^{1}}$ is defined by
\[l([a,b]):=\mu_L([a,b])=\left\{\begin{array}{ll}\hat{b}-\hat{a}, & \mbox{if \,\,
$0\leq\hat{a}<\hat{b}<1$},\\1+\hat{b}-\hat{a}, & \mbox{if \,\,
$0\leq{\hat{b}}<\hat{a}<1$}.\end{array}\right.\]

\begin{defi}\label{condition C} {\rm Let  $R_1>1$ and $\varepsilon>0$ be constants.
The points $x_{0}, z_{i}\in {S^{1}}$
 with
$z_1\prec{z}_2\prec{z}_3\prec{z}_4\prec{z}_1$
  satisfy\textbf{ conditions $(C_{R_1,\varepsilon})$} if:
\begin{itemize} \item[(a)]
$R^{-1}_{1}l([z_{2},z_{3}])\sqrt{\varepsilon}\leq l([z_{1},z_{2}])\leq{R_{1}l([z_{2},z_{3}])\sqrt[4]{\varepsilon}}$;
 \item[(b)]
$R^{-1}_{1}l([z_{2},z_{3}])\leq l([z_{3},z_{4}])\leq R_{1}l([z_{2},z_{3}])$;
\item[(c)] \
 $\underset{1\leq i\leq 4}{\max}l([x_{0},z_{i}])\leq
R_{1}l([z_{2},z_{3}])$.
\end{itemize}}
\end{defi}

 For $x_{0}\in {S^{1}}$ with lift $\hat{x}_{0}$ in $[0,1)$ define
$d(\hat{x}_0):=\min\{\hat{x}_0,(1-\hat{x}_0)\}.$

\begin{lemm}\label{lemm4.2} Assume, that the lift $\hat{\psi}$  of the conjugating map $\psi$
 has a positive derivative $D\hat{\psi}(\hat{x}_0)=\omega$ at the point $\hat{x}_0\in [0,1)$
 and let $R_1>1$ be a constant. Then there exists a constant $C_{2}=C_{2}(\omega,R_{1})$
such that for any $\varepsilon>0$ there exists $\delta=\delta(\hat{x}_0,\varepsilon)\in (0,d(\hat{x}_0))$
such that for all $z_{i}\in S^1$ with $\hat{z}_i \in (\hat{x}_{0}-\delta,\hat{x}_{0}+\delta),$ $i=1,2,3,4$, satisfying
the conditions $(C_{R_1,\varepsilon})$ one has:
\begin{itemize} \item[i)]
$\frac{l([z_{1},z_{2}])}{l([z_{2},z_{3}])}(1-C_{2}\sqrt{\varepsilon})
\leq{\frac{l[\psi(z_{1}), \psi(z_{2})]}{l[\psi(z_{2}),\psi(z_{3})]} \leq
\frac{l([z_{1},z_{2}])}{l([z_{2},z_{3}])}(1+C_{2}\sqrt\varepsilon)}$,
\item[ii)]$\frac{l([z_{3},z_{4}])}{l([z_{2},z_{3}])}(1-C_{2}\varepsilon)\leq{\frac{l[\psi(z_{3}),
\psi(z_{4})]}{l[\psi(z_{2}), \psi(z_{3})]}\leq \frac{l([z_{3},z_{4}])}{l([z_{2},z_{3}])}
(1+C_{2}\varepsilon)}.$
\end{itemize}
\end{lemm}

\begin{lemm}\label{lemm4.3}
Suppose the lift $\hat{\psi}$ has a positive derivative $D\hat{\psi}(\hat{x}_{0})=\omega$ at
 the point $\hat{x}_{0}\in{[0,1)}$ and let $R_1>1$ be a constant. Then
there exists a constant $R_{2}=R_{2}(\omega,R_{1})$
such that for any $\varepsilon>0$
there exists $\delta=\delta(\hat{x}_0,\varepsilon)\in (0,d(\hat{x}_0))$
such that for all $z_{i}\in S^1$ with $\hat{z}_i\in (\hat{x}_{0}-\delta,\ \hat{x}_{0}+\delta),$ $i=1,2,3,4,$ satisfying
the conditions $(C_{R_1,\varepsilon})$ one has:
\begin{eqnarray}
|Dist(z_1,z_2,z_3,z_4;\psi))-1|\leq {R_2\sqrt{\varepsilon}.}
\end{eqnarray}

\end{lemm}

The main idea for proving that the  map  $\psi$ conjugating $f_{1}$ and $f_{2}$ is a singular
function is
to construct a quadruple of points $z_i,\, i=1,2,3,4$, for which the ratio of the distortions
 $Dist (z_{1},z_{2},z_{3},z_{4};f^{q_{n}}_{1})$ and
$Dist (\psi({z}_{1}),\psi({z}_{2}),\psi({z}_{3}),\psi({z}_{4});f^{q_{n}}_{2})$ stays away from 1.

For this assume  $D\hat{\psi}(\hat{x}_{0})=\omega>0$ for the lift $\hat{x}_{0}\in
[0,1)$ of a point $x_{0}\in S^{1}.$ W.l.o.g. we can choose $n$ to be odd.
Then we have
$\Delta_{0}^{(n)}(z)=[f_{1}^{q_{n}}(z),z]$ and
$\Delta_{0}^{(n-1)}(z)=[z,f^{q_{n-1}}_{1}(z)]$ for any point $z$ of the circle.
 Consider the $n$-th dynamical partition $\xi_{n}(x_0)$ of the  point $x_0\in S^1$
defined by the homeomorphism $f_{1}$. Only one interval of the partition  $\xi_{n}(x_0)$
covers the break point $a_{1}$. Hence there exists an unique point $\overline {a}_{1}$ with
either $\overline {a}_{1}\in \Delta_{0}^{(n-1)}(x_{0})$ and  $f^{l}_{1}(\overline {a}_{1})=a_{1}$
for some $0 \leq l < q_ {n}$,  or $\overline {a}_{1}\in \Delta_{0}^{(n)}(x_{0})$ and $f^{l}_{1}(\overline {a}_{1})=a_{1}$
 for some $0 \leq l < q_ {n-1}$. We call the point $\overline {a}_{1}$ the \textbf{$q_{n}$-preimage}
of the break point ${a}_{1}$ in $ \Delta_{0}^{(n-1)}(x_{0}) \cup \Delta_{0}^{(n)}(x_{0}).$
There exists an unique point $t_{0}$ such that $\overline{a}_{1}$ is the \textbf{middle point} of the interval
$[ t_{0},f^{q_{n-1}}_{1}(t_{0})]$ (see Figure 1).
\begin{figure}
\centering
\includegraphics[width=14cm]{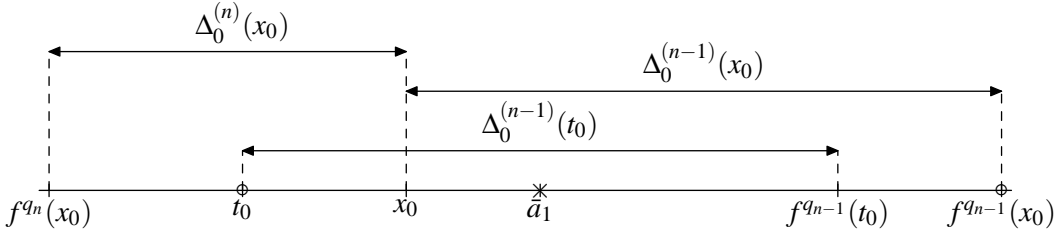}
\caption{The point $\overline{a}_{1}=f_{1}^{-l} (a_{1})$
belongs to the interval $[ f_{1}^{q_{n}}(x_{0}),f_{1}^{q_{n-1}}(x_{0})]$
and is the middle point of $[ t_{0},f_{1}^{q_{n-1}} (t_{0})].$}
\end{figure}
Consider now the $n$-th dynamical partitions $\xi_{n}(t_{0})$ of the  point $t_{0}$
defined by $f_{1}$ on the first circle respectively $\zeta_{n}(\psi(t_{0}))$ of the point $\psi(t_{0})$
 defined by $f_{2}$ on the second circle. For each $n\geq 1$ define
 $$
 \Delta^{(n)}_{i}(t_{0}):=f^{i}_{1}(\Delta^{(n)}_{0}(t_{0})),
\,\,\ C^{(n)}_{i}(\psi(t_{0})):=f^{i}_{2}(C^{(n)}_{0}(\psi(t_{0})),\,\,0\leq i < q_{n+1},$$
 where
$\Delta^{(n)}_{0}(t_{0})$ respectively $C^{(n)}_{0}(\psi(t_{0}))$ are the initial intervals of
order $n$ of the points $t_{0}$ respectively $\psi(t_{0})$ determined by $f_1$ respectively
$f_2$ . By definition

$$\xi_{n}(t_{0})=\{\Delta^{(n-1)}_{i}(t_{0}), \,\, 0\leq i< q_n \} \cup \{
\Delta^{(n)}_{j}(t_{0}), \,\, 0\leq{j}< q_{n-1}\},$$
$$\,\,\,\,\,\,\,\,\,\,\,\,\,\,\,\,\,\,\,\,\zeta_{n}(\psi(t_{0}))=\{{C^{(n-1)}_{i}(\psi(t_{0})) ,\,\, 0\leq i < q_n
\}\cup \{C^{(n)}_{j}(\psi(t_{0})), 0\leq{j}<q_{n-1}}\}.$$

 Since the common rotation number $\rho$ of $f_{1}$ and $f_{2}$ is
irrational,  the order of the points on the orbit \,\,
$\{f^{k}_{1}(x_{0}), \, k\in\mathbb{Z} ^1\} $ on the first circle will be precisely
the same as the one  for the orbit \,\, $\{f^k_{2}(\psi(x_0)),\,\, k\in \mathbb{Z}^1\}$
on the second circle. This together with the relation
 $\psi(f_1 (x))=f_2(\psi (x))$ for $x\in S^{1}$ implies that
$$\psi\Big(\Delta_{i}^{(n-1)}(t_{0})\Big)=C_{i}^{(n-1)}(\psi(t_{0})),\,\, 0\leq i<
q_{n},\,\,\ \psi\Big(\Delta_{j}^{(n)}(t_{0})\Big)=C_{j}^{(n)}(\psi(t_{0})),\ \ 0\leq
j< q_{n-1}.$$
Denote by $\overline {b}_{1}$ the $q_{n}$-preimage  of the second break point ${b}_{1}$ of $f_1$ in $ \Delta_{0}^{(n-1)}(t_{0}) \cup \Delta_{0}^{(n)}(t_{0}),$
such that $f_{1}^{p}(\overline {b}_{1})=b_{1}$ for some $0\leq p < q_{n}.$
 The $\psi$-preimages of the points
 $\overline{a}_{1}$ and $\overline{b}_{1}$ lie in
 $C^{(n-1)}_{0} (\psi(t_{0}))\cup C^{(n)}_{0}(\psi(t_{0}))$.
Using  relation (5) we get
$$
f_{2}^{l}(\psi(\overline{a}_1))=f_{2}^{l-1}
(f_{2}(\psi(\overline{a}_1)))=
f_{2}^{l-1}(\psi(f_{1}(\overline{a}_1)))=...=
\psi(f_{1}^{l}(\overline{a}_1)) =\psi(a_{1})=a_{2}.
$$
Similarly one shows $f_{2}^{p}(\psi(\overline{b}_1))=b_{2}$.

For $\varepsilon>0$ define the two neighbourhoods $U_{n}, V_{n}$ of the point $\overline{a}_{1}\in S^1$ as

$$U_{n}(\overline{a}_{1})=\{z\in S^1: \hat{z}\in  (\hat{\overline{a}}_{1}-\delta_{n},\, \hat{\overline{a}}_{1}+\delta_{n}) \,\, \text{with}\,\,\, \delta_n =
\frac{1}{4}l([t_{0},f^{q_{n-1}}_{1}(t_{0})])\sqrt{\varepsilon}\},$$
$$V_{n}(\overline{a}_{1})=\{z\in S^1: \hat{z}\in (\hat{\overline{a}}_{1}-\gamma_{n},\, \hat{\overline{a}}_{1}+\gamma_{n})\,\, \text{with} \,
\,\,\gamma_n=\frac{1}{2}l([t_{0},f^{q_{n-1}}_{1}(t_{0})])\sqrt[4]{\varepsilon}\}.$$
It is clear that $U_{n}(\overline{a}_{1})\subset
V_{n}(\overline{a}_{1})\subset [t_{0},f^{q_{n-1}}_{1}(t_{0})]$
(see Figure 2).
\begin{figure}
\centering
\includegraphics[width=14cm]{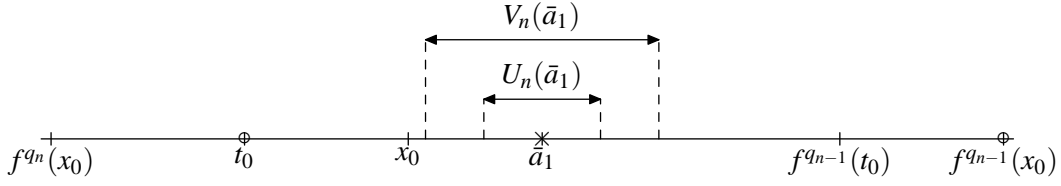}
\caption{The intervals $U_{n}(\overline{a}_{1})$ and $V_{n}(\overline{a}_{1})$
are $\sqrt{\varepsilon}$ and $\sqrt[4]{\varepsilon}$ comparable with
$[t_{0},f_{1}^{q_{n-1}}(t_{0})]$ respectively.}
\end{figure}

Then two cases are possible:

$\textbf{either} \,\,\, \overline{b}_{1} \in U_{n}(\overline{a}_{1})\,\,\,
\textbf{or}\,\,\,\overline{b}_{1} \notin U_{n}(\overline{a}_{1})\,\, i.e.\,\,
\overline{b}_{1} \in \left[f^{q_{n}}(t_{0}), f^{q_{n-1}}(t_{0})\right]\setminus U_{n}(\overline{a}_{1}).
$\\
Consider first the case   $\overline{b}_{1} \in U_{n}(\overline{a}_{1}).$
  If $\hat{\overline{b}}_{1}$ lies on the left hand side of the point  $\hat{\overline{a}}_{1}$  we define
\begin{eqnarray}\label{eq46}
\hat{z}_{1}:=\hat{\overline{a}}_{1}-\frac{1}{2} l([t_{0},f^{q_{n-1}}(t_{0})])\sqrt[4]{\varepsilon}, \,\
\hat{z}_{2}:=\hat{\overline{a}}_{1},\nonumber\\
\hat{z}_{3}:=\hat{\overline{a}}_{1}+\frac{1}{4}l([t_{0},f^{q_{n-1}}(t_{0})]),\,\,\
\hat{z}_{4}:=\hat{\overline{a}}_{1}+\frac{1}{2}l([t_{0},f^{q_{n-1}}(t_{0})]),
\end{eqnarray}
corresponding to the points  $z_{i}\in S^{1},\,i=1,2,3,4$, with $z_{2}=\overline {a}_{1}$ and
$z_{1}\prec z_{2}\prec z_{3}\prec  z_{4} \prec z_{1}. $
If on the other hand  $\hat{\overline{b}}_{1}$ is on right hand side of $\hat{\overline{a}}_{1},$ we define
\begin{eqnarray}\label{eq400006}
\hat{z}_{1}:=\hat{\overline{a}}_{1}-\frac{1}{2}l([t_{0},f^{q_{n-1}}(t_{0})]), \,\
\hat{z}_{2}:=\hat{\overline{a}}_{1}-\frac{1}{4}l([t_{0},f^{q_{n-1}}(t_{0})]),\nonumber\\
\hat{z}_{3}:=\hat{\overline{a}}_{1}, \,\,\
\hat{z}_{4}:= \hat{\overline{a}}_{1}+\frac{1}{2} l([t_{0},f^{q_{n-1}}(t_{0})])\sqrt[4]{\varepsilon},
\end{eqnarray}
corresponding to  the points  $z_{i}\in S^{1},\,i=1,2,3,4$, with $z_{3}=\overline {a}_{1}$ and
$z_{1}\prec z_{2}\prec z_{3}\prec  z_{4} \prec z_{1}$.  in the following we consider only the first case, the second one can be handled similarly.
Then one shows
\begin{lemm}\label{lemm4.4} Suppose the circle homeomorphism $f_1$
satisfies the conditions of Lemma \ref{lem2.1}. Let $\delta>0$ be the constant determined by Lemma \ref{lemm4.2} and
let for large enough $n$ the points  $\hat{z}_i,\, i=1,2,3 4$  defined  in (\ref{eq46}) be the lifts of the points $z_{i}\in S^{1}, i=1,2,3,4$.  Then the triple of
intervals $[z_s, z_{s+1}], s=1,2,3 $  has the following
properties:
\begin{itemize}
\item[(1)] $[z_1,z_4]$, $[f^{q_n}_1(z_1),f^{q_n}_1 (z_4)]\subset U_\delta (x_0) =\{z\in S^1: \hat{x} \in
(\hat{x}_0-\delta, \hat{x}_0+\delta)$;
\item[(2)]
the intervals $[z_s,z_{s+1}]$, $[f^{q_n}_1(z_s),f^{q_n}_1
(z_{s+1})]$, \,\,$s=1,2,3,$  satisfy conditions $(C_{R_1,\varepsilon})$
for some constant $R_{1}>1$
depending only on
the variation $v$  of $\log Df_1$.
\end{itemize}
\end{lemm}

\begin{lemm}\label{lemm4.5} Assume the circle
homeomorphisms $f_{i},\,\,i=1,2,$ satisfy the conditions of
Theorem \ref{ADM1}. Let  $z_i\in S^1, \ i=1,2,3,4,$ be the points defined in Lemma \ref{lemm4.4}.
 Then the following inequalities hold for
sufficiently large $n$:
\begin{eqnarray}\label{eq48}
\Big|  Dist (z_{1},z_{2},z_{3},z_{4};f_{1}^{q_{n}})-
\sigma_{f_{1}}(a_{1}) \cdot\sigma_{f_{1}}(b_{1})\Big|\leq R_{2}\sqrt[4]{\varepsilon,}\\
\Big|  Dist (\psi(z_{1}),\psi(z_{2}),\psi(z_{3}),\psi(z_{4});f_{2}^{q_{n}})-
\sigma_{f_{2}}(a_{2}) \cdot\sigma_{f_{2}}(b_{2})\Big|\leq R_{2}\sqrt[4]{\varepsilon,}
\end{eqnarray}

where  the positive  constant $R_{2}=R_{2}(R_{1},f_{1},f_{2})$  does not depend on $\varepsilon.$
\end{lemm}

After these preparations we can now proceed to the  proof of Theorem \ref{ADM1}.\\

\textbf{Proof of Theorem \ref{ADM1}}. Let $f_1$ and $f_2$ be circle
homeomorphisms satisfying the conditions of Theorem \ref{ADM1}.
  The lift $\hat{\psi}(\hat{x})$ of the conjugating map
$\psi(x)$ is a continuous and monotone increasing function on $R^{1}$.
 Hence $\hat{\psi}(\hat{x})$ has a finite
 derivative $D\hat{\psi}(x)$
 almost everywhere (w.r.t. Lebesgue measure) on $R^{1}$.
 Recall that  $D\hat{\psi}(\hat{x}+1)=D\hat{\psi}(\hat{x})$
for each
  $\hat{x}\in R^{1}$ where the  derivative $D\hat{\psi}(\hat{x} )$
is defined.
 It is enough to  show that $D\hat{\psi}(\hat{x})=0$ for almost all
 points $\hat{x}$ of the interval $[0,1).$ Suppose $D\hat{\psi}(\hat{x}_{0})=\omega>0$ for
some point $\hat{x}_{0}\in [0, 1)$ corresponding to the point $x_{0}\in S^{1}.$
Choose an $\varepsilon>0$ and the points $z_i \in S^1 , \,\,\, i=1,2,3,4$, with lifted
 vector $( \hat{z}_{1},\hat{z}_{2}, \hat{z}_{3},  \hat{z}_{4})$ as
 defined in (\ref{eq46}). Then by the second assertion of Lemma \ref{lemm4.4}
the intervals  $[z_s,z_{s+1}]$, $[f^{q_n}_1(z_s),f^{q_n}_1
(z_{s+1})]$, \,\,$s=1,2,3,$  satisfy conditions $(C_{R_1,\varepsilon})$ for
some constant $R_{1}>1$
depending only on
the variation $v$ of $\log Df_{1}$.

 Lemma \ref{lemm4.3} next implies
\begin{eqnarray}\label{eqn49}
|Dist (z_{1}, z_{2},z_{3},z_{4};\psi)-1|\leq
R_{2}\sqrt{\varepsilon}
\end{eqnarray}
and
\begin{eqnarray}\label{eqn410}
|Dist (f_{1}^{q_{n}}(z_{1}),
f_{1}^{q_{n}}(z_{2}),f_{1}^{q_{n}}(z_{3}),f_{1}^{q_{n}}(z_{4});\psi)-1|\leq
R_{2}\sqrt{\varepsilon}.
\end{eqnarray}
Hence
\begin{eqnarray}\label{eqn411}
\Big|\frac{Dist (f_{1}^{q_{n}}(z_{1}),
f_{1}^{q_{n}}(z_{2}),f_{1}^{q_{n}}(z_{3}),f_{1}^{q_{n}}(z_{4});\psi)}{Dist (z_{1},
z_{2},z_{3},z_{4};\psi)}-1\Big|\leq R_{3}\sqrt{\varepsilon},
\end{eqnarray}
where the constant $R_{3}>0$ does not depend on $\varepsilon$ and
$n$.\\
But by definition

$Dist (f_{1}^{q_n}(z_{1}),f_{1}^{q_n}(z_{2}),
f_{1}^{q_n}(z_{3}),f_{1}^{q_n}(z_{4}); \psi)=$

$$=\frac{Cr (\psi(f_{1}^{q_{n}}(z_{1})),
\psi(f_{1}^{q_{n}}(z_{2})),\psi(f_{1}^{q_{n}}(z_{3})),\psi(f_{1}^{q_{n}}(z_{4})))}{Cr (f_{1}^{q_n}(z_{1}),f_{1}^{q_n}(z_{2}),
f_{1}^{q_n}(z_{3}),f_{1}^{q_n}(z_{4}))}.$$
 Since $\psi$ is
conjugating $f_1$ and $f_2$ we can readily see that

$Cr (\psi(f_{1}^{q_{n}}(z_{1})),
\psi(f_{1}^{q_{n}}(z_{2})),\psi(f_{1}^{q_{n}}(z_{3})),\psi(f_{1}^{q_{n}}(z_{4})))=$

$$=Cr (f_{2}^{q_{n}}(\psi(z_{1})),
f_{2}^{q_{n}}(\psi(z_{2})),f_{2}^{q_{n}}(\psi(z_{3})),f_{2}^{q_{n}}(\psi(z_{4}))).$$

It now follows that  $\,\,\,$

$$ \frac{Dist (f_{1}^{q_{n}}(z_{1}),
f_{1}^{q_{n}}(z_{2}),f_{1}^{q_{n}}(z_{3}),f_{1}^{q_{n}}(z_{4});\psi)}{Dist (z_{1},
z_{2},z_{3},z_{4};\psi)}=$$

$$=\frac{Cr (\psi(f_{1}^{q_{n}}(z_{1})),
\psi(f_{1}^{q_{n}}(z_{2})),\psi(f_{1}^{q_{n}}(z_{3})),\psi(f_{1}^{q_{n}}(z_{4})))}{Cr (f_{1}^{q_{n}}(z_{1}),
f_{1}^{q_{n}}(z_{2}),f_{1}^{q_{n}}(z_{3}),f_{1}^{q_{n}}(z_{4}))}\times
\frac{Cr (z_{1}, z_{2},z_{3},z_{4})}{Cr (\psi(z_{1}),\psi(
z_{2}),\psi(z_{3}),\psi(z_{4}))}=$$

$$=\frac{Cr (f_{2}^{q_{n}}(\psi(z_{1})),
f_{2}^{q_{n}}(\psi(z_{2})),f_{2}^{q_{n}}(\psi(z_{3})),f_{2}^{q_{n}}(\psi(z_{4})))}{Cr (\psi(z_{1}),\psi(
z_{2}),\psi(z_{3}),\psi(z_{4}))}:\frac{Cr (f_{1}^{q_{n}}(z_{1}),
f_{1}^{q_{n}}(z_{2}),f_{1}^{q_{n}}(z_{3}),f_{1}^{q_{n}}(z_{4}))}{Cr (z_{1},
z_{2},z_{3},z_{4})}\nonumber\\ =$$

 $$=\frac{Dist (\psi(z_{1}),\psi(
z_{2}),\psi(z_{3}),\psi(z_{4});f_{2}^{q_{n}})}{Dist (z_{1},
z_{2},z_{3},z_{4};f_{1}^{q_{n}})}.\nonumber\\
$$

Combining this with inequality (\ref{eqn411}) we get
\begin{eqnarray}\label{eqn412}
\Big|\frac{Dist (\psi(z_{1}),\psi(
z_{2}),\psi(z_{3}),\psi(z_{4});f_{2}^{q_{n}})}{Dist (z_{1},
z_{2},z_{3},z_{4};f_{1}^{q_{n}})}-1\Big|\leq
R_{3}\sqrt{\varepsilon}.
\end{eqnarray}

But using  Lemma \ref{lemm4.5} we get
\begin{eqnarray}\label{eqn413}
\Big|\frac{Dist (\psi(z_{1}),\psi(z_{2}),\psi(z_{3}),\psi(z_{4});
f_{2}^{q_{n}})}{Dist (z_{1},
z_{2},z_{3},z_{4};f_{1}^{q_{n}})}-\frac{\sigma_{f_{2}}(a_{2})
 \cdot\sigma_{f_{2}}(b_{2}) }{\sigma_{f_{1}}(a_{1}) \cdot\sigma_{f_{1}}(b_{1}) }\Big|
\leq Const\sqrt[4]{\varepsilon}
\end{eqnarray}
for sufficiently large $n$. This contradiction proves Theorem \ref{ADM1} in the first case.\\
There remains the case where the  point $\overline{b}_{1}$
belongs to the set
 $\left[f^{q_{n}}_{1}(t_{0}), f^{q_{n-1}}_{1}(t_{0})\right]\setminus U_{n}(\overline{a}_{1}).$

Let  $\hat{\overline{b}}_{1}$ lie on the left hand side of the point  $\hat{\overline{a}}_{1}$, 
the case $\hat{\overline{b}}_{1}$  on the right hand side of $\hat{\overline{a}}_{1}$
 can be  handled similarly. We define
\begin{eqnarray}\label{eq146}
\hat{z}_{1}:=\hat{\overline{a}}_{1}-\frac{1}{4} l([t_{0},f^{q_{n-1}}(t_{0})])\sqrt{\varepsilon}, \,\
\hat{z}_{2}:=\hat{\overline{a}}_{1},\nonumber\\
\hat{z}_{3}:=\hat{\overline{a}}_{1}+\frac{1}{4}l([t_{0},f^{q_{n-1}}(t_{0})]),
\hat{z}_{4}:=\hat{\overline{a}}_{1}+\frac{1}{2}l([t_{0},f^{q_{n-1}}(t_{0})]),
\end{eqnarray}
which determine  the points  $z_{i}\in S^{1},\,i=1,2,3,4$ with $z_{2}=\overline {a}_{1}$ and
$z_{1}\prec z_{2}\prec z_{3}\prec  z_{4} \prec z_{1}. $
The proof of Theorem \ref{ADM1} for the corresponding intervals
$[z_{s},z_{s+1}]$,\, $s=1,2,3$,
proceeds now exactly as in the previous case. This concludes the proof of Theorem \ref{ADM1}.\\

\section{ Proofs of Lemmas \ref{lemm4.2} - \ref{lemm4.5}.}

We start with the proof of Lemma \ref{lemm4.2}.
\begin{proof}
Suppose, the
derivative $D\hat{\psi}(\hat{x}_{0})$ exists and $D\hat{\psi}(\hat{x}_{0})=\omega>0$ for the lift $\hat{x}_{0}\in [0,1)$
of some point $x_{0}$ in $ S^{1}$.
 By the definition of the derivative there exists for any
$\varepsilon>0$ a number \,\,
$\delta=\delta(x_{0},\varepsilon)\in{(0,d(x_{0}))}$ such that, for
all $\hat{x}\in (\hat{x}_{0}-\delta, \ \hat{x}_{0}+\delta),$
\begin{eqnarray}\label{eqn27}
\omega-\varepsilon<\frac{\hat{\psi}(\hat{x})-\hat{\psi}(\hat{x}_{0})}{\hat{x}-\hat{x}_{0}}<\omega+\varepsilon.\end{eqnarray}
Now take four points $\hat{z}_{i} \in (\hat{x}_{0}-\delta, \ \hat{x}_{0}+\delta)\subset [0, 1)$
satisfying conditions $(C_{R_1,\varepsilon})$. W.l.o.g.
we can assume that $[{z}_1,{z}_4]\subset U_\delta (x_{0})$
and $z_{1}\prec z_{4}\prec x_{0}\prec z_{1}$. Relation
(\ref{eqn27}) then implies for $\hat{x}=\hat{z}_{i},$ $ i=1,2,3,4$
$$
(\omega-\varepsilon)(\hat{x}_{0}-\hat{z}_{i})<\hat{\psi}(\hat{x}_{0})-\hat{\psi}(\hat{z}_{i})<
(\omega +\varepsilon)(\hat{x}_{0}-\hat{z}_{i}).
$$

This yields the following inequalities for $\hat{z}_{s}, \, s=1,2,3$

\begin{eqnarray}\label{eqn28}
 \omega-\varepsilon\frac{(\hat{x}_{0}-\hat{z}_{s+1})+(\hat{x}_{0}
-\hat{z}_{s})}{\hat{z}_{s+1}-\hat{z}_{s}}&<&
\frac{\hat{\psi}(\hat{z}_{s+1})-\hat{\psi}(\hat{z}_{s})}{\hat{z}_{s+1}-\hat{z}_{s}}\nonumber\\
&<&
\omega+\varepsilon
\frac{(\hat{x}_{0}-\hat{z}_{s+1})+(\hat{x}_{0}-\hat{z}_{s})}{\hat{z}_{s+1}-\hat{z}_{s}},
\end{eqnarray}
respectively for $s=1,2$
\begin{eqnarray}\label{eqn29}
\omega -\varepsilon\frac{(\hat{x}_{0}-\hat{z}_{s+2})+(\hat{x}_{0}-\hat{z}_{s})}{\hat{z}_{s+2}-\hat{z}_{s}}&\leq&
\frac{\hat{\psi}(\hat{z}_{s+2})-\hat{\psi}(\hat{z}_{s})}{\hat{z}_{s+2}-\hat{z}_{s}}\nonumber\\
&\leq&\omega+\varepsilon\frac{(\hat{x}_{0}-\hat{z}_{s+2})+(\hat{x}_{0}-\hat{z}_{s})}{\hat{z}_{s+2}-\hat{z}_{s}}.
\end{eqnarray}
Since the points $z_{i}, i=1,2,3,4$,  satisfy conditions $(C_{R_1,\varepsilon}),$ it is easy to show
 that
\begin{eqnarray}\label{eqn30}
\underset{1\leq i\leq
4}{\max}\Big\{\frac{\hat{x}_{0}-\hat{z}_{i}}{\hat{z}_{2}-\hat{z}_{1}}\Big\}\leq{R_{1}\frac{\hat{z}_{3}-\hat{z}_{2}}{\hat{z}_{2}-\hat{z}_{1}}\leq\frac{R^{2}_{1}}
{\sqrt{\varepsilon}}},
\end{eqnarray}
\begin{eqnarray}\label{eqn31}
\underset{1\leq i\leq
4}{\max}\Big\{\frac{\hat{x}_{0}-\hat{z}_{i}}{\hat{z}_{3}-\hat{z}_{2}},\frac{\hat{x}_{0}-\hat{z}_{i}}{\hat{z}_{4}-\hat{z}_{3}}\Big\}\leq{R^{2}_{1}}.
\end{eqnarray}
Combining relations (\ref{eqn28}), (\ref{eqn29}), (\ref{eqn30}) and
 (\ref{eqn31})
 we get

\begin{eqnarray}\label{eqn32}
\omega-C_{4}\sqrt{\varepsilon}\leq
\frac{\hat{\psi}(\hat{z}_{2})-\hat{\psi}(\hat{z}_{1})}{\hat{z}_{2}-\hat{z}_{1}}\leq
\omega+C_{4}\sqrt{\varepsilon};
\end{eqnarray}
  for $l=2,3$ we get
\begin{eqnarray}\label{eqn33}
\omega-C_{4}\varepsilon\leq\frac{\hat{\psi}(\hat{z}_{l+1})-\hat{\psi}(\hat{z}_{l})}{\hat{z}_{l+1}-\hat{z}_{l}}\leq\omega+C_{4}\varepsilon,
\end{eqnarray}
respectively for $s=1,2$

\begin{eqnarray}\label{eqn34}
\omega-C_{4}\varepsilon\leq\frac{\hat{\psi}(\hat{z}_{s+2})-\hat{\psi}(\hat{z}_{s})}{\hat{z}_{s+2}-\hat{z}_{s}}\leq\omega+C_{4}\varepsilon,
\end{eqnarray}
 where the constant $C_{4}>0$ depends on $R_{1}$, $\omega$, but does
not depend on $l([\hat{z}_{s},\hat{z}_{s+1}]),s=1,2,3$, and on $\varepsilon$.
Using the equality
$$\frac{\hat{\psi}(\hat{z}_{s+1})-\hat{\psi}(\hat{z}_{s})}{\hat{\psi}(\hat{z}_{s})-\hat{\psi}(\hat{z}_{s-1})}:\frac{\hat{z}_{s+1}-\hat{z}_{s}}{\hat{z}_{s}-\hat{z}_{s-1}}=
\frac{\hat{\psi}(\hat{z}_{s+1})-\hat{\psi}(\hat{z}_{s})}{\hat{z}_{s+1}-\hat{z}_{s}}\cdot\frac{\hat{z}_{s}-\hat{z}_{s-1}}{\hat{\psi}(\hat{z}_{s})-
\hat{\psi}(\hat{z}_{s-1})}$$
and relations (\ref{eqn32}),(\ref{eqn33}),(\ref{eqn34}) we get the
assertions of Lemma \ref{lemm4.2}.
\end{proof}

 Next we will prove Lemma \ref{lemm4.3}.\\
\begin{proof}
 Since
$$Dist (z_{l},z_{2},z_{3},z_{4},\psi)=
\frac{Cr (\hat{\psi}(\hat{z}_{1}),\hat{\psi}(\hat{z}_{2}),\hat{\psi}(\hat{z}_{3}),\hat{\psi}(\hat{z}_{4}))}
{Cr (\hat{z}_{1},\hat{z}_{2},\hat{z}_{3},\hat{z}_{4})}=
$$
$$=\frac{\hat{\psi}(\hat{z}_{2})-\hat{\psi}(\hat{z}_{1})}{\hat{z}_{2}-\hat{z}_{1}}\cdot
\frac{\hat{\psi}(\hat{z}_{4})-\hat{\psi}(\hat{z}_{3})}{\hat{z}_{4}-\hat{z}_{3}}\cdot
\frac{\hat{z}_{3}-\hat{z}_{1}}{\hat{\psi}(\hat{z}_{3})-\hat{\psi}(\hat{z}_{1})}
\cdot\frac{\hat{z}_{4}-\hat{z}_{2}}{\hat{\psi}(\hat{z}_{4})-
\hat{\psi}(\hat{z}_{2})}$$
inequalities (\ref{eqn32})-(\ref{eqn34}) prove Lemma  \ref{lemm4.3}.
\end{proof}
We continue with the proof of Lemma  \ref{lemm4.4}.
\begin{proof}
 We assume $n$ to be odd.
 Hence $f^{q_{n}}_{1}(z)\prec z \prec f^{q_{n-1}}_{1}(z)\prec f^{q_{n}}_{1}(z)$
for any point $z$ on the circle $S^{1}.$ The point $\overline{a}_{1}$ lies in the interval
$[f^{q_{n}}_{1}(x_{0}),f^{q_{n-1}}_{1}(x_{0})]$ and  is the middle point of the
interval $[t_{0},f^{q_{n-1}}_{1}(t_{0})].$ This and the structure of the orbits imply
$x_{-3q_{n-1}}\prec f^{q_{n}}_{1}(t_{0}) \prec t_{0} \prec f^{q_{n-1}}_{1}(t_{0})\prec x_{3q_{n-1}}. $
By construction $[z_{1},z_{4}]\subset[t_{0},f^{q_{n-1}}_{1}(t_{0})]. $
Consequently
 $[f^{q_{n}}_{1}(z_{1}),f^{q_{n}}_{1}(z_{4})]\subset[f^{q_{n}}_{1}(t_{0}),f^{q_{n}+q_{n-1}}_{1}(t_{0})]. $
Summarizing we get therefore
\begin{eqnarray}\label{eqn71}
[z_{1},z_{4}],[f^{q_{n}}_{1}(z_{1}),f^{q_{n}}_{1}(z_{4})]\subset [x_{-3q_{n-1}}, x_{3q_{n-1}} ].
\end{eqnarray}
Obviously
$$[x_{-3q_{n-1}}, x_{3q_{n-1}} ]=
[x_{-3q_{n-1}}, x_{-2q_{n-1}}] \cup [x_{-2q_{n-1}}, x_{-q_{n-1}}]
\cup[x_{-q_{n-1}}, x_{0}]\cup$$
\begin{eqnarray}\label{e71}
\cup[x_{0}, x_{q_{n-1}}]\cup[x_{q_{n-1}}, x_{2q_{n-1}}]\cup[x_{2q_{n-1}}, x_{3q_{n-1}}].
\end{eqnarray}
By Corollary \ref{cor3} the intervals
 $[x,y]$, $[f^{q_{n-1}}_{1}(x),f^{q_{n-1}}_{1}(y)]$ and
$[f^ {-q_{n-1}}_{1}(x),f^{-q_{n-1}}_{1}(y)]$
are $e^{v_{1}}$-comparable for any $x,y\in S^{1}. $ This together with  equation (\ref{e71})  and Corollary \ref{cor2}
 implies  that
$$
l([x_{-3q_{n-1}}, x_{3q_{n-1}} ])\leq (1+5e^{3v_{1}})l([x_{0}, x_{q_{n-1}}])\leq const \lambda^{n}_{1},
$$
for a constant $\lambda_{1} \in (0,1)$. For sufficiently large $n $ then obviously
$[x_{-3q_{n-1}}, x_{3q_{n-1}}]\subset (x_{0}-\delta, \, x_{0}+\delta)$.
This together with (\ref{eqn71}) implies the first assertion of Lemma \ref{lemm4.4}.

Next we will prove the second assertion of Lemma \ref{lemm4.4} .

By Corollary \ref{cor3} the intervals $[z_{s},z_{s+1}] $ and $[f^ {q_{n}}(z_{s}),f^{q_{n}}(z_{s+1})] $
 are $e^{v_{1}}$-comparable for all $s=1,2,3.$ Using the definition of the points $z_{s}, s=1,2,3,4$ and
the Denjoy inequality it is easy to verify that these intervals satisfy the assumptions a) and b)
 of conditions $(C_{R_1,\varepsilon})$\,. Using the relations
$$
[z_{1},z_{4}],[f^{q_{n}}_{1}(z_{1}),f^{q_{n}}_{1}(z_{4})]\subset [x_{-3q_{n-1}}, x_{3q_{n-1}}]
$$
we get
\begin{eqnarray}\label{eqq71}
\underset{1\leq s \leq
4}{\max} \Big\{ |\hat{x}_{0}-\hat{z}_{s}|, \, |\hat{x}_{0}-\hat{y}_{s}|\Big\} \leq
l([x_{-3q_{n-1}}, x_{3q_{n-1}}])
\end{eqnarray}
where ${y}_{s}:=f^{q_{n}}_{1}(z_{s})\,\ s=1,2,3,4.$
Now we want to compare the lengths of the intervals $[x_{-3q_{n-1}}, x_{3q_{n-1}}]$
and $[t_{0},f^{q_{n-1}}_{1}(t_{0})].$
Using the definition of $t_{0}$ it is easy to see
that $x_{-2q_{n-1}}\prec t_{0}\prec x_{2q_{n-1}}.$  Applying  
$f^{sq_{n-1}}_{1},\,s\in \mathbb{Z} $, to these relations we get
 $x_{(s-2)q_{n-1}}\prec t_{sq_{n-1}}\prec x_{(s+2)q_{n-1}},$
\,\,$s\in  \mathbb{Z}$.
In particular the last relations imply
$ t_{-5q_{n-1}}\prec x_{-3q_{n-1}},\,\,x_{3q_{n-1}} \prec t_{5q_{n-1}}$
and hence
$ [x_{-3q_{n-1}},\,x_{3q_{n-1}}] \subset [t_{-5q_{n-1}},\, t_{5q_{n-1}}].$
Consequently
\begin{eqnarray}\label{eqqq71}
l[(x_{-3q_{n-1}},\,x_{3q_{n-1}}])\leq l([t_{-5q_{n-1}},\, t_{5q_{n-1}}]).
\end{eqnarray}
But the intervals $[t_{-5q_{n-1}},t_{5q_{n-1}}]$ and
$[t_{0},f^{q_{n-1}}_{1}(t_{0})]$ are
$(1+2e^{v_{1}}+ 2e^{2v_{1}}+2e^{3v_{1}}+2e^{4v_{1}}+ e^{5v_{1}})$-comparable.
This together with eq. (\ref{eqq71}) implies
$l[(x_{-3q_{n-1}},\,x_{3q_{n-1}}])\leq 10 e^{5v_{1}}l([t_{0},f^{q_{n-1}}_{1}(t_{0})]).$
Finally, we conclude that the points $z_{s},\,s=1,2,3,4$ and $ f^{q_{n}}_{1}(z_{s}),\,s=1,2,3,4, $
satisfy conditions $(C_{R_1,\varepsilon})$ with the constant $R_{1}=40e^{5v_{1}}$.
This concludes the proof of Lemma \ref{lemm4.4}.
\end{proof}

Remains the proof of Lemma \ref{lemm4.5}.
\begin{proof}
Choose the points $z_s$,
$s=1,2,3,4$, according to formulas (\ref{eq46}) and consider the two
sets of intervals
  $\{\,f_{1}^{i}[z_{s}, z_{s+1}],$ \, $0\leq i <q_{n},$ \, $s=1,2,3\, \}$ and $\{\,f_{2}^{i}[\psi(z_s), \psi(z_{s+1})],$ \,\,$ 0\leq i <q_{n},$ \,\, $s=1,2,3\, \}$.
   By the construction of the intervals $[z_{s}, z_{s+1}],$ $s=1,2,3$, only the intervals
  $f_{1}^{l}([z_1,z_2])$
and
 $f_{1}^{p}([z_1,z_2])$
cover the break points $a_{1}$  respectively $b_{1}$, namely
$a_{1}=f_{1}^{l}(z_{2})$,\, $b_{1}\in f_{1}^{p}[z_{1},z_{2}).$

Similarly, alone the intervals
$f_{2}^{l}[\psi(z_1),\psi(z_2)]$
respectively
$f_{2}^{p}[\psi(z_1), \psi(z_2)]$
cover the break points $a_2$  respectively $b_2$, namely  $a_2=f_{2}^{l}(\psi(z_{2}))$ and $b_2 \in f_{2}^{p}[\psi(z_{1}),\psi(z_{2})).$
 Next we compare the distortions
$Dist (z_1,z_2,z_3,z_4;f_{1}^{q_{n}})$
 and
$Dist (\psi(z_{1}),\psi(z_{2}),\psi(z_{3}),\psi(z_{4});f_{2}^{q_{n}}).$
We estimate only the first distortion, the second one can be estimated analogously.
Rewriting it as
$$Dist (z_1,z_2,z_3,z_4;f_{1}^{q_{n}})= \prod_{\substack {i=0 \\
i\neq l,p }}^{q_n-1}
Dist (f_{1}^{i}(z_{1}),f_{1}^{i}(z_{2}),f_{1}^{i}(z_{3}),f_{1}^{i}(z_{4});
f_{1})\times$$
\begin{eqnarray}\label{eq66} \,\,\ \,\ \times Dist (f_{1}^{l}(z_{1}),f_{1}^{l}(z_{2}),f_{1}^{l}(z_{3}),f_{1}^{l}(z_{4});
f_{1})\times Dist (f_{1}^{p}(z_{1}),f_{1}^{p}(z_{2}),f_{1}^{p}(z_{3}),f_{1}^{p}(z_{4});
f_{1}),\end{eqnarray}

we apply Lemma \ref{lemm3.1} to obtain $$\,\,\prod_{\substack {i=0 \\ i\neq
l,p}}^{q_{n}-1} Dist (f_{1}^{i}(z_{1}),
f_{1}^i(z_{2}),f_{1}^i(z_{3}),f_{1}^{i}(z_{4});f_{1})=$$
\begin{eqnarray}\label{eq68}
=\exp\{\sum_{\substack {i=0 \\ i\neq l,p }}^{q_{n}-1}
\text{log}(1+O(|[f_{1}^{i}(z_{1}),f_{1}^{i}(z_4)]|^{1+\alpha}))\}.
\end{eqnarray}

By construction 
$[f_{1}^{i}(z_1),f_{1}^{i}(z_4)]\subset
[f_{1}^{i}(t_{0}),f_{1}^{i}(f_{1}^{q_{n-1}}(t_{0}))]$ for $ 0\leq i < q_n$.
 By Corollary \ref{cor2} the length of the last interval is bounded by
$\text{const} \,\lambda^{n}_{1}$. Thus we get for \,\, $ 0\leq i < q_n$
\begin{eqnarray}\label{eq69}
l([f_{1}^{i}(z_1),f_{1}^{i}(z_4)])\leq const \, \lambda^{n}_{1}.
\end{eqnarray}
 The interval $[t_{0},f_{1}^{q_{n-1}}(t_{0})]$ is $q_{n}$-small and hence
\begin{eqnarray}\label{eq70}
\sum_{i=0}^{q_{n}-1}l([f_{1}^{i}(z_1),f_{1}^{i}(z_4)])\leq 1.
\end{eqnarray}
 Combining equations (\ref{eq68}), (\ref{eq69}), (\ref{eq70}) we get
$$\big|\prod_{\substack {i=0 \\ i\neq l,p }}^{q_n-1}Dist (f^{i}_{1}(z_{1}),f^{i}_{1}(z_{2}),
f^{i}_{1}(z_{3}),f^{i}_{1}(z_{4});f_{1})-1\big|\leq$$
 \begin{eqnarray}\label{eq71}
 \leq const\,
\lambda^{n\alpha}_{1}\sum_{\substack {i=0 \\ i\neq l,p
}}^{q_{n}-1}l([f_{1}^{i}(z_1),f_{1}^{i}(z_4)])\leq const\,
\lambda^{n\alpha}_{1}.\end{eqnarray}
 Next we estimate the distortions
$$Dist (f^{l}_{1}(z_{1}),f^{l}_{1}(z_{2}),
f^{l}_{1}(z_{3}),f^{l}_{1}(z_{4});f_1)  \,\, \text{and} \,\,
Dist (f^{p}_{1}(z_{1}),f^{p}_{1}(z_{2}),
f^{p}_{1}(z_{3}),f^{p}_{1}(z_{4}); f_1).$$
Define for $0\leq m < q_{n}$ the length ratios
\begin{eqnarray}\label{eq72}
\vartheta(m):=\frac{l([f^{m}_{1}(z_{2}),f^{m}_{1}(z_{3})])}{l([f^{m}_{1}(z_{1}),f^{m}_{1}(z_{2})])},\,\,\,
\,\,\,
\tau(m):=\frac{l([f^{m}_{1}{(\bar{b}_{1})},f^{m}_{1}(z_{2})])}{l([f^{m}_{1}(z_{1}),f^{m}_{1}(z_{2})])}.
\end{eqnarray}
Lemma \ref{lemm2.2} then implies the following inequalities  \begin{eqnarray}\label{eq73}
e^{-v_{1}} \cdot\vartheta(0)\leq \vartheta(m) \leq e^{v_{1}}\cdot \vartheta(0),\,\,\,
e^{-v_{1}}\cdot\tau(0)\leq \tau(m)\leq e^{v_{1}}\cdot\tau(0).
\end{eqnarray}
Using the definitions of the points $z_{i},i=1,2,3$ we get
\begin{eqnarray}\label{eq74}
\vartheta(0)=\frac{1}{2\sqrt[4]{\varepsilon}},\,\,\,\,
0\leq \tau(0)\leq \frac{\sqrt[4]{\varepsilon}}{4}.
\end{eqnarray}

Define next for $x>0$ and $0\leq t\leq 1$ the functions $G(x, \sigma)$ and $F(x,t, \sigma)$ as
\begin{eqnarray}\label{eq75}
G(x,\sigma):=\frac{\sigma(1+x)}{\sigma +x},\,\,
F(x,t,\sigma):=\frac{[\sigma+(1-\sigma)t](1+x)}{\sigma+(1-\sigma)t+x}.
\end{eqnarray}

 Applying
Lemma \ref{lemm3.4} we get
\begin{eqnarray}\label{eq54}
|Dist (f^{l}_{1}(z_{1}),f^{l}_{1}(z_{2}),
f^{l}_{1}(z_{3}),f^{l}_{1}(z_{4});f_{1})-
G(\vartheta(l),\sigma_{f_{1}}(a_{1}))|\leq
K_{2} \,l([f^{l}_{1}(z_{1}),f^{l}_{1}(z_{4})]),
\end{eqnarray}
$$|Dist (f^{p}_{1}(z_{1}),f^{p}_{1}(z_{2}),
f^{p}_{1}(z_{3}),f^{p}_{1}(z_{4});f_{1})-
F(\vartheta(p),\tau(p),\sigma_{f_{1}}(b_{1}))|\leq$$
\begin{eqnarray}\label{eq76}
\,\,\,&&  K_{2} \,l([f^{p}_{1}(z_{1}),f^{p}_{1}(z_{4})]).
\end{eqnarray}
The definitions of the functions $G$,  $F$ together with  equations (\ref{eq72}) and (\ref{eq73}) imply
\begin{eqnarray}\label{eq77}
\mid G(\vartheta(l),\sigma_{f_{1}}(a_{1}))-G(\vartheta(0),\sigma_{f_{1}}(a_{1}))\mid
 \leq \,K_{3} \sqrt[4]{\varepsilon}, \\ \nonumber
\mid F(\vartheta(p),\tau(p),\sigma_{f_{1}}(b_{1}))
- F(\vartheta(0),\tau(0),\sigma_{f_{1}}(b_{1}))
\mid
 \,\leq K_{3} \sqrt[4]{\varepsilon},
\end{eqnarray}
\begin{eqnarray}\label{eq78}
\mid G(\vartheta(0),\sigma_{f_{1}}(a_{1}))-\sigma_{f_{1}}(a_{1})\mid
 \leq K_{3} \sqrt[4]{\varepsilon},\\ \nonumber
\mid F(\vartheta(0),\tau(0),\sigma_{f_{1}}(b_{1})
-\sigma_{f_{1}}(b_{1})\mid
 \leq K_{3} \sqrt[4]{\varepsilon},
\end{eqnarray}
where the constant $K_{3}$ is given by $$K_{3}=\left( \sigma_{f_{1}}(a_{1})|1-\sigma_{f_{1}}(a_{1})|+
\sigma_{f_{1}}(b_{1})|1-\sigma_{f_{1}}(b_{1})|\right)(1+e^{v_{1}}).$$
The last four equations imply
\begin{eqnarray}\label{eq79}
\mid G(\vartheta(l),\sigma_{f_{1}}(a_{1}))-\sigma_{f_{1}}(a_{1})\mid
 \leq 2\, K_{3} \sqrt[4]{\varepsilon},
 \end{eqnarray}
 \begin{eqnarray}\label{eq80}
\mid F(\vartheta(p),\tau(p),\sigma_{f_{1}}(b_{1})
-\sigma_{f_{1}}(b_{1})\mid
 \leq 2\, K_{3} \sqrt[4]{\varepsilon}.
\end{eqnarray}

Combining equations (\ref{eqqq71}), (\ref{eq66}), (\ref{eq70}) and (\ref{eq75})-(\ref{eq80})
we obtain finally

$$
\Big|  Dist (z_{1},z_{2},z_{3},z_{4};f_{1}^{q_{n}})-
\sigma_{f_{1}}(a_{1}) \cdot\sigma_{f_{1}}(b_{1})\Big|\leq R_{2}\sqrt[4]{\varepsilon,}
$$
which proves the first inequality in Lemma \ref{lemm4.5}.  The second inequality 
can be proven by using similar arguments as above.
\end{proof}
\section{Acknowledgements}
This work has been supported by the German Research Foundation (DFG) through the grant MA 633/22-1.  A. Dzhalilov is grateful also to  the Abdus Salam International Center for Theoretical Physics (ICTP) for partial financial support. We thank Fredrik Str\"omberg for his help  with the eps-files with the figures.


\begin{thebibliography}{24}
\bibitem{Ar1961} V.I. Arnol'd, Small denominators: I. Mappings from the
circle onto itself. Izv. Akad. Nauk SSSR, Ser. Mat.,
\,\textbf{25},\,21-86\,(1961).

\bibitem{De1932} A. Denjoy, Sur les courbes d\'{e}finies par les \'{e}quations
diff\'{e}rentielles \`{a} la surface du tore. J. Math. Pures
Appl.,\, \textbf{11},\,333-375\, (1932).

\bibitem{dFdM1999} E. de Faria and W. de Melo, Rigidity of critical circle mappings.
I. J. Eur. Math. Soc. (JEMS),\,\textbf{1},\,(4),\,339-392\,(1999).

\bibitem{DK1998} A.A. Dzhalilov and K.M. Khanin, On invariant measure
for homeomorphisms of a circle with a point of break., Funct.
Anal. Appl.,\, \textbf{32},\,(3)\,153-161\,(1998).


\bibitem{DL2006} A.A. Dzhalilov and I. Liousse, Circle homeomorphisms
with two break points.
Nonlinearity,\, \textbf{19},\, 1951-1968\,(2006).

\bibitem{DLM2009} A.A. Dzhalilov, I. Liousse and
D. Mayer, Singular measures of piecewise smooth circle  homeomorphisms
with two break points.
Discrete and continuous dynamical systems,\, \textbf{24}, (2),\, 381-403\,(2009).

\bibitem{DHT2010} A.A. Dzhalilov, H. Akin, S. Temir,
Conjugations between circle maps with a single break point.
Journal of Mathematical Analysis and Applications,
\, \textbf{366}, \, 1-10\,(2010).

\bibitem{CFS1982} I.P. Cornfeld, S.V. Fomin and Ya.G. Sinai, Ergodic Theory,
Springer Verlag, Berlin\, (1982).


\bibitem{He1979} M. Herman, Sur la conjugaison diff\'{e}rentiable des
diff\'{e}omorphismes du cercle \`{a} des rotations. Inst. Hautes
Etudes Sci. Publ. Math., \,\textbf{49}, \,225-234\,(1979).


\bibitem{KO1989} Y. Katznelson and D. Ornstein, The absolute continuity
of the conjugation of certain diffeomorphisms of the circle.
Ergod. Theor. Dyn. Syst.,\,\textbf{9},\,681-690 \,(1989).

\bibitem{KS1989} K.M. Khanin and Ya.G. Sinai, Smoothness of conjugacies
of diffeomorphisms of the circle with rotations. Russ. Math.
Surv., \,\textbf{44},\,69-99\,(1989), translation of Usp. Mat.
Nauk, \,\textbf{44},\,57-82\,(1989).

\bibitem{KV1991} K.M. Khanin and E.B. Vul, Circle homeomorphisms with weak
discontinuities. Advances in Soviet Mathematics,
\,\textbf{3},\,57-98\, (1993).

\bibitem{KhKm2003} K.M. Khanin and D. Khmelev, Renormalizations and
Rigidity Theory
for Circle Homeomorphisms with Singularities of the Break Type.
Commun. Math. Phys.,\,\textbf{235},\,69-124\,(2003).

\bibitem{Li2005} I. Liousse, Nombre de rotation, mesures invariantes et
ratio set des hom\'{e}omorphisms
affines par morceaux du cercle, Ann.Inst. Fourier,\,
\textbf{55},\, 431-482\,(2005).

\bibitem{Mo1966} J. Moser, A rapid convergent iteration method and
non-linear differential equations. II. Ann. Scuola Norm. Sup.
Pisa, \,\textbf{20}(3),\,499-535\,(1966).

\bibitem{dMvS1993} W. de Melo and S. van Strien, One dimensional
Dynamics. Springer Verlag Berlin,\,
p.3-25\, (1993).


\bibitem{St1992} M. Stein, Groups of piecewise linear homeomorphisms.
Trans. A.M.S. \textbf{32},\, 477-514,\, (1992).

\bibitem{TK2004} A.Yu. Teplinskii
and K.M. Khanin, Rigidity for circle diffeomorphisms with
simgularities. Russ. Math.
Surv.,\,\textbf{759}(2),\,329-353.\,(2004), translation of Usp.
Mat. Nauk, \,\textbf{59}(2),\,137-160\,(2004).

\bibitem{Yo1984} J.C. Yoccoz, Il n'y a pas de contre-exemple de Denjoy
analytique. C. R. Acad. Sci. Paris T. 298; 7, 141-144 (1984).


\end{thebibliography}
\end{document}